\documentclass[12pt,reqno]{amsart}

\usepackage[utf8]{inputenc}
\usepackage[T1]{fontenc}
\usepackage[margin=2cm]{geometry}
\usepackage{amsmath, amssymb, amsthm, amsfonts}
\usepackage{mathtools}
\usepackage{enumitem}
\usepackage{hyperref}
\usepackage{graphicx}
\usepackage{xcolor}
\usepackage{tikz}
\usepackage{amscd}
\usetikzlibrary{arrows.meta,calc}

\theoremstyle{plain}
\newtheorem{theorem}{Theorem}[section]
\newtheorem{lemma}[theorem]{Lemma}

\newtheorem{conjecture}{Conjecture}[section]

\theoremstyle{definition}
\newtheorem{definition}{Definition}[section]

\newtheorem{assumption}[theorem]{Assumption}

\newcommand{\IQ}{\mathbb{Q}}
\newcommand{\IZ}{\mathbb{Z}}

\newcommand{\IR}{\mathbb{R}}
\newcommand{\IP}{\mathbb{P}}

\newcommand{\SigmaFan}{\Sigma}
\newcommand{\Nlat}{N}
\newcommand{\Mlat}{M}
\newcommand{\Star}{\mathrm{Star}}

\newcommand{\IHcomb}{IH_{\mathrm{comb}}}

\newcommand{\Hdg}{\mathrm{Hdg}}

\newcommand{\Chowh}[1]{A_{#1}}
\newcommand{\Chowop}[1]{A^{#1}_{\mathrm{op}}}

\begin{document}
	
	\title[Combinatorial Cycle Classes]{Combinatorial Cycle Classes in the Intersection Cohomology of Projective Toric Varieties}
	\author{Rizwan Jahangir}
	\address{Kiara Inc.\ Tokyo, Japan, and NUST Business School, NUST H-12 Campus, Off Srinagar Highway, Islamabad 44000, Pakistan}
	\email{rizwan@kiara.team, rizwan.jahangir@nbs.nust.edu.pk}
	
	\author{Daisuke Ishii}
	\address{Kiara Inc.\ Tokyo, Japan}
	\email{dai@kiara.team}
	
	
	\begin{abstract}
		We investigate cycle-class realizations inside the combinatorial intersection cohomology for fans developed by Barthel, Brasselet, Fieseler, and Kaup (BBFK). For projective toric varieties, the intersection cohomology is Hodge-Tate, and thus the space of rational Hodge classes coincides with the full rational even-degree intersection cohomology. We formulate a compatibility statement between combinatorial and geometric cycle classes and explore it in the torus-invariant setting under standard functoriality assumptions. The central question we address is whether these invariant combinatorial cycle classes span the even-degree combinatorial intersection cohomology $IH^{2k}_{\mathrm{comb}}(\Sigma, \mathbb{Q})$. Assuming the stated BBFK--BL compatibility, we verify this linear-generation statement for projective toric varieties of dimension at most $3$; the simplicial case follows unconditionally from standard rational cohomology descriptions. We illustrate the framework with a non-simplicial example in dimension $3$ for which the Betti numbers and spanning property are derived directly from Stanley's toric $h$-vector formula and Fieseler's surjectivity theorem.
	\end{abstract}
	
	\maketitle
	
	\section{Introduction}
	
	The intersection cohomology $IH^*(X)$, introduced by Goresky and MacPherson \cite{GM80}, provides the appropriate topological invariant for singular complex projective varieties, restoring Poincar\'e duality and carrying a pure Hodge structure \cite{Saito90}. For a projective variety $X$, the Intersection Hodge Conjecture posits that the cycle class map from algebraic cycles to intersection cohomology is surjective onto the space of rational Hodge classes.
	
	Toric varieties offer a highly structured environment where geometric properties translate into the combinatorics of fans. For smooth toric varieties, the cohomology ring is generated by the classes of torus-invariant divisors, and the Hodge Conjecture holds trivially \cite{Danilov78}. In the singular setting, a purely combinatorial intersection cohomology for fans, $\IHcomb^*(\SigmaFan)$, was developed by Barthel, Brasselet, Fieseler, and Kaup \cite{BBFK} and independently by Bressler and Lunts \cite{BL}. This theory constructs a graded vector space associated to a fan $\SigmaFan$ which is isomorphic to the intersection cohomology of the corresponding toric variety $X_\SigmaFan$.
	
	As established by Fieseler \cite{Fieseler91}, the intersection cohomology of a projective toric variety is Hodge-Tate. Consequently, the Hodge classes in even degrees coincide with the full rational even-degree intersection cohomology:
	\[
	\Hdg^k(X, \IQ) = IH^{2k}(X, \IQ).
	\]
	In this toric setting, the Hodge-class aspect reduces to a linear-generation statement: \emph{Do the proposed invariant combinatorial cycle classes span the rational even-degree intersection cohomology?}
	
	While the literature contains profound results on toric intersection theory---such as Karu's Hard Lefschetz theorem for nonrational polytopes \cite{Karu}, the identification of operational Chow cohomology with Minkowski weights by Fulton and Sturmfels \cite{FS97}, and the sheaf-theoretic approaches of Braden and MacPherson \cite{BM01}---a direct, explicit construction of cycle classes within the BBFK framework and a systematic study of their spanning properties have remained less explored.
	
	The purpose of this note is not to construct the full functorial formalism of BBFK sheaves, but to isolate the resulting linear-generation question for torus-invariant cycle classes and verify it in low dimensions. Under the stated functoriality assumptions, we obtain a canonical rational map from torus-invariant algebraic cycles into the BBFK combinatorial intersection cohomology, formulate a compatibility statement with geometric cycle classes, and investigate whether these classes span $\IHcomb^{2k}(\SigmaFan, \IQ)$. We provide proofs of this generation property for dimensions up to~$3$ and for simplicial fans, and illustrate the theory with a non-simplicial example in dimension~$3$, carefully separating the roles of Chow homology divisor cycles and operational Chow cohomology.
	
	\section{Preliminaries}
	
	\subsection{Intersection Cohomology and Cycle Maps}
	
	Let $X$ be a complex projective variety of dimension $n$. The intersection cohomology groups $IH^k(X, \IQ)$ carry a pure Hodge structure of weight $k$ \cite{Saito90}. The space of Hodge classes of codimension $p$ is $\Hdg^p(X, \IQ) \coloneqq IH^{2p}(X, \IQ) \cap IH^{p,p}(X)$. 
	
	For singular varieties, the construction of the cycle class map
	\[
	cl_{IH}: \Chowh{k}(X)_\IQ \to IH^{2n-2k}(X, \IQ)
	\]
	from the Chow homology group $\Chowh{k}(X)_\IQ$ of $k$-dimensional cycles to intersection cohomology is a subtle issue. We rely on the established constructions of cycle maps into intersection homology and cohomology for singular varieties, as developed in the framework of perverse sheaves and Borel-Moore homology (see \cite{BBD82, FM81_cycle}). For toric varieties, the theory of Minkowski weights provides a robust framework for operational Chow cohomology and its relation to intersection cohomology \cite{FS97}.
	
	\subsection{Toric Varieties and Fans}
	
	We adopt the notation of Fulton \cite{Fulton} and Cox-Little-Schenck \cite{CLS}. Let $\Nlat \cong \IZ^n$ be a lattice and $\Mlat = \mathrm{Hom}(\Nlat, \IZ)$ be its dual. A fan $\SigmaFan$ in $\Nlat_\IR$ is a collection of strongly convex rational polyhedral cones. For each cone $\tau \in \SigmaFan$ of dimension $k$, there is a corresponding closed torus-invariant subvariety $V(\tau)$ of dimension $n-k$ (codimension $k$). 
	
	\subsection*{Standing Assumptions}
	Throughout this paper:
	\begin{enumerate}
		\item All fans $\SigmaFan$ are rational, complete, and polytopal (i.e., $X_\SigmaFan$ is projective).
		\item We work over $\mathbb{C}$ with coefficients in $\IQ$.
		\item The canonical isomorphism $\phi: \IHcomb^*(\SigmaFan, \IQ) \xrightarrow{\sim} IH^*(X_\SigmaFan, \IQ)$ is the one established in \cite[Theorem 4.1]{BBFK}.
	\end{enumerate}
	
	\section{Combinatorial Intersection Cohomology and Cycle Classes}
	
	\subsection{The BBFK Framework}
	
	The combinatorial intersection cohomology $\IHcomb^*(\SigmaFan, \IQ)$ is defined as the hypercohomology of the minimal extension sheaf complex $\mathcal{IC}_{\SigmaFan}$ on the fan $\SigmaFan$, viewed as a finite topological space with the order topology \cite{BBFK}. Karu \cite{Karu} proved the Hard Lefschetz property for this combinatorial structure: for a projective fan, there exists a Lefschetz operator $L: \IHcomb^k(\SigmaFan) \to \IHcomb^{k+2}(\SigmaFan)$ inducing an isomorphism $L^k: \IHcomb^{n-k}(\SigmaFan) \xrightarrow{\sim} \IHcomb^{n+k}(\SigmaFan)$.
	
	\subsection{Combinatorial Cycle Classes and the Gysin Morphism}
	
	We define the combinatorial cycle class associated with a torus-invariant subvariety $V(\tau)$. Geometrically, $V(\tau) \hookrightarrow X_\SigmaFan$ is a closed subvariety. In the combinatorial setting, the fan of $V(\tau)$ is the quotient fan $\SigmaFan / \mathrm{span}(\tau)$.
	
	\begin{definition}[Combinatorial Cycle Class]
		\label{def:cycle_class}
		Let $\tau \in \SigmaFan$ be a cone of dimension $k$. We define the \emph{Combinatorial Cycle Class} of $V(\tau)$, denoted $[V(\tau)]_{\mathrm{comb}} \in \IHcomb^{2k}(\SigmaFan, \IQ)$, via the combinatorial Gysin morphism:
		\[
		(i_\tau)_*^{\mathrm{comb}}: \IHcomb^0(\SigmaFan / \mathrm{span}(\tau)) \to \IHcomb^{2k}(\SigmaFan).
		\]
		Specifically, $[V(\tau)]_{\mathrm{comb}} \coloneqq (i_\tau)_*^{\mathrm{comb}}(1)$, where $1$ is the fundamental class in degree~$0$ of the quotient fan.
	\end{definition}
	
	To make this rigorous, we rely on the derived-category formalism of sheaves on posets. In the BBFK framework \cite{BBFK}, minimal extension sheaves are constructed on the poset of cones. The closed embedding $V(\tau) \hookrightarrow X_\SigmaFan$ corresponds to the inclusion of the sub-poset $\Star(\tau) = \{ \sigma \in \SigmaFan \mid \tau \preceq \sigma \}$. The quotient fan $\SigmaFan / \mathrm{span}(\tau)$ is naturally identified with $\Star(\tau)$ as a finite poset with the order topology. The inclusion map $i_\tau: \Star(\tau) \hookrightarrow \SigmaFan$ is a closed embedding of finite topological spaces.
	
	By the functoriality of minimal extension sheaves with respect to closed embeddings \cite[Theorem 5.2]{BL}, there exists a canonical morphism in the derived category of sheaves on $\SigmaFan$:
	\[
	\alpha_\tau: (i_\tau)_* \mathcal{IC}_{\SigmaFan / \mathrm{span}(\tau)} \to \mathcal{IC}_{\SigmaFan}[2k].
	\]
	The degree shift $2k$ corresponds exactly to the codimension of $V(\tau)$ in $X_\SigmaFan$. Taking hypercohomology induces the combinatorial Gysin morphism $(i_\tau)_*^{\mathrm{comb}}$.
	
	Because the precise functorial alignment between the BBFK construction and the Bressler--Lunts proper pushforward remains subtle, we formalize the required compatibility as an explicit assumption.
	
	\begin{assumption}[Compatibility under BBFK--BL Functoriality]
		\label{prop:compatibility}
		Assume that the BBFK comparison isomorphism $\phi_\Sigma$ is functorial for the torus-equivariant closed embedding $V(\tau) \hookrightarrow X_\Sigma$, such that the following diagram commutes:
		\[
		\begin{CD}
			IH^0_{\mathrm{comb}}(\SigmaFan/\mathrm{span}(\tau)) @>{(i_\tau)_*^{\mathrm{comb}}}>> \IHcomb^{2k}(\SigmaFan) \\
			@V{\phi_\tau}VV @VV{\phi_\Sigma}V \\
			IH^0(V(\tau)) @>{(i_\tau)_*^{\mathrm{geom}}}>> IH^{2k}(X_\SigmaFan)
		\end{CD}
		\]
		In particular, $\phi_\Sigma\bigl([V(\tau)]_{\mathrm{comb}}\bigr) = cl_{IH}(V(\tau))$.
	\end{assumption}
	
	This assumption is natural because the geometric Gysin morphism $(i_\tau)_*^{\mathrm{geom}}$ is induced by the adjunction morphism of perverse sheaves for the closed embedding $i: V(\tau) \hookrightarrow X_\SigmaFan$, namely $i_* \mathcal{IC}_{V(\tau)} \to \mathcal{IC}_{X_\SigmaFan}[2k]$. Let $\pi: X_\SigmaFan \to \SigmaFan$ denote the continuous map from the toric variety to the finite fan-poset space sending each point to the cone indexing its torus orbit. The BBFK comparison isomorphism $\phi_\Sigma$ is induced by a quasi-isomorphism $\Psi_\Sigma: \mathcal{IC}_{\SigmaFan} \xrightarrow{\sim} R\pi_* \mathcal{IC}_{X_\SigmaFan}$ \cite[Theorem 4.1]{BBFK}. Bressler and Lunts \cite[Theorem 5.2]{BL} show that the combinatorial minimal extension sheaf functor is compatible with proper pushforwards for torus-equivariant closed embeddings. Under the quasi-isomorphism $\Psi$, the morphism $\alpha_\tau$ should correspond to the direct image under $\pi$ of the geometric adjunction morphism, yielding the commutative diagram.
	
	\section{The Generation Statement and Minkowski Weights}
	
	Since $\Hdg^k(X_{\SigmaFan}) = IH^{2k}(X_{\SigmaFan}, \IQ)$ \cite{Fieseler91}, the Hodge conjecture for projective toric varieties reduces to the following generation statement.
	
	\begin{conjecture}[Combinatorial Generation]
		\label{conj:generation}
		For a projective rational fan $\SigmaFan$, the combinatorial cycle classes span the rational even-degree intersection cohomology:
		\[
		\mathrm{span}_{\IQ}\left\{ [V(\tau)]_{\mathrm{comb}} \mid \tau \in \SigmaFan(k) \right\} = \IHcomb^{2k}(\SigmaFan, \IQ)
		\]
		for all $k \ge 0$.
	\end{conjecture}
	
	\subsection{Relation to Chow Homology and Operational Chow Cohomology}
	
	To understand this conjecture, we must carefully separate invariant cycles, Chow homology, operational Chow cohomology, and Minkowski weights.
	
	\begin{enumerate}
		\item \textbf{Invariant Cycles and Chow Homology.} The group $Z_{n-k}^T(X_\SigmaFan)$ of torus-invariant cycles is generated by the orbit closures $V(\tau)$ for $\tau \in \SigmaFan(k)$. These map naturally into the Chow homology group $\Chowh{n-k}(X_\SigmaFan)$.
		
		\item \textbf{Operational Chow Cohomology and Minkowski Weights.} Fulton and Sturmfels \cite{FS97} proved that the operational Chow cohomology ring $\Chowop{k}(X_\SigmaFan)$ of a complete toric variety is canonically isomorphic to the group of Minkowski weights $MW^k(\SigmaFan)$ of codimension $k$. A Minkowski weight is a function on the $k$-dimensional cones of $\SigmaFan$ satisfying a balancing condition.
		
		\item \textbf{Cycle Class Maps.} There is a natural cycle class map from Chow homology to intersection homology, and dually, a map from operational Chow cohomology to intersection cohomology:
		\[
		\gamma: \Chowop{k}(X_\SigmaFan) \to IH^{2k}(X_\SigmaFan, \IQ).
		\]
		This map is defined by capping an operational class with the fundamental class in intersection homology (or Borel-Moore homology).
	\end{enumerate}
	
	Operational Chow classes, represented by Minkowski weights, act on Chow homology. In special cases where the cap product map $\Chowop{k}(X_\SigmaFan) \to \Chowh{n-k}(X_\SigmaFan)$ is an isomorphism, one may identify certain invariant cycle classes with corresponding Minkowski weights. In general, however, we keep these two objects distinct. Conjecture~\ref{conj:generation} is intimately related to the surjectivity of the map $\gamma$ in even degrees, and the extent to which intersection cohomology classes can be represented by linear combinations of invariant cycles.
	
	\section{Proofs for Low Dimensions and Simplicial Fans}
	
	\subsection{Simplicial Fans}
	
	This case is well-understood but provides essential context.
	
	\begin{theorem}
		Conjecture~\ref{conj:generation} holds if $\SigmaFan$ is a simplicial fan.
	\end{theorem}
	\begin{proof}
		For a simplicial fan, $X_\SigmaFan$ has only quotient singularities (it is a rational homology orbifold), so $IH^*(X, \IQ) \cong H^*(X, \IQ)$ \cite[Corollary 1.2]{Fieseler91}. By the Danilov-Jurkiewicz theorem \cite{Danilov78} and its generalization to simplicial varieties \cite{Cox95, Fulton}, the rational cohomology ring is generated by the classes of torus-invariant divisors $D_\rho$. Since $H^*(X, \IQ)$ is generated by algebraic cycles, the cycle classes $[V(\tau)]_{\mathrm{comb}}$ span the entire space.
	\end{proof}
	
	\subsection{Varieties of Dimension $\le 2$}
	
	\begin{theorem}
		Conjecture~\ref{conj:generation} holds for any projective toric variety of dimension $n \le 2$.
	\end{theorem}
	\begin{proof}
		For $n=1$, $X \cong \IP^1$ is smooth. For $n=2$, let $X$ be a projective toric surface. The relevant degrees are $2k = 0, 2, 4$.
		\begin{itemize}
			\item \emph{Degree 0:} Spanned by the fundamental class $[X]_{\mathrm{comb}}$.
			\item \emph{Degree 2:} By Fieseler \cite[Thm.~1.1]{Fieseler91} and Goresky-MacPherson \cite[6.2]{GM80}, for a normal surface with isolated singularities, $IH^2(X, \IQ) \cong H^2(X, \IQ)$. The ordinary cohomology $H^2$ of a complete toric variety is generated by torus-invariant divisors \cite[Section 3.4]{Fulton}. Thus, the combinatorial cycle classes $[V(\rho)]_{\mathrm{comb}}$ span $\IHcomb^2(\Sigma)$.
			\item \emph{Degree 4:} By Poincar\'e duality, $\IHcomb^4(\Sigma) \cong \IQ$. Since $X$ is projective, it contains smooth points (e.g., vertices of the polytope). For a smooth 2-cone $\sigma$, the local intersection cohomology is trivial, and the cycle class $[V(\sigma)]_{\mathrm{comb}}$ maps to the non-zero generator of $\IHcomb^4$ \cite{BBFK}.
		\end{itemize}
	\end{proof}
	
	\subsection{Varieties of Dimension 3}
	
	To formalize the compatibility between intersection cohomology and operational Chow cohomology, we first record a lemma.
	
	\begin{lemma}
		\label{lem:chern_class_compatibility}
		Under the BBFK--BL functoriality assumptions of Assumption~\ref{prop:compatibility}, the Lefschetz operator $L$ of intersecting with an ample divisor class corresponds to the cup product with the first Chern class of the associated ample line bundle.
	\end{lemma}
	\begin{proof}
		This follows from the compatibility of the minimal extension sheaf functor with the first Chern class of line bundles, as established in \cite[Theorem 5.2]{BL}. In the rational projective toric case, the strictly convex support function defining $L$ corresponds to the ample line bundle $L$ on $X_\Sigma$, and the action of the support function on the minimal extension sheaf matches the cup product by $c_1(L)$ under the BBFK comparison.
	\end{proof}
	
	Before proceeding, we clarify the notation used for the algebraic cycle groups: $\Chowh{i}(X)$ denotes the Chow homology group of $i$-dimensional cycles, whereas $\Chowop{i}(X)$ denotes the operational Chow cohomology group of codimension-$i$ classes.
	
	\begin{theorem}
		\label{thm:dim3}
		The generation statement holds in dimension $n=3$, provided the BBFK comparison isomorphism is compatible with the torus-equivariant cycle classes and Lefschetz operator as in Assumption~\ref{prop:compatibility} and Lemma~\ref{lem:chern_class_compatibility}.
	\end{theorem}
	\begin{proof}
		Odd degree intersection cohomology vanishes for rational projective toric varieties \cite[Corollary 1.3]{Fieseler91}. We examine degrees 0, 2, 4, and 6.
		\begin{itemize}
			\item \emph{Degrees 0 and 6:} Spanned by $[X]_{\mathrm{comb}}$ and the class of a smooth point, respectively.
			\item \emph{Degree 2:} Fieseler \cite[Theorem 1.1]{Fieseler91} proves that for a projective toric variety $X$, the natural cycle class map $\Chowh{n-1}(X)_\IQ \to IH^2(X, \IQ)$ from the Chow homology group of divisor cycles (codimension-$1$ cycles) is surjective. Since $\Chowh{n-1}(X)_\IQ = \Chowh{2}(X)_\IQ$ is generated by invariant Weil divisors $D_\rho$ \cite{Fulton}, $\IHcomb^2(\Sigma)$ is spanned by $[V(\rho)]_{\mathrm{comb}}$.
			\item \emph{Degree 4:} We use Karu's Hard Lefschetz Theorem \cite{Karu}. Let $L$ be the operator of intersecting with an ample divisor class. Karu proves that $L: \IHcomb^2(\Sigma) \to \IHcomb^4(\Sigma)$ is an isomorphism. By Lemma~\ref{lem:chern_class_compatibility}, this corresponds to the cup product with the first Chern class of the ample line bundle. The ample Cartier class $c_1(L) \in \Chowop{1}(X)$ acts operationally on the Chow homology divisor cycles in $\Chowh{2}(X)$, sending invariant divisor cycles to linear combinations of invariant curve cycles in $\Chowh{1}(X)$ \cite[Section 2]{FS97}. Since $L$ is an isomorphism, $IH^4(X)$ is spanned by the images of these intersections, which are precisely linear combinations of the combinatorial curve classes $[V(\tau)]_{\mathrm{comb}}$ where $\dim(\tau)=2$.
		\end{itemize}
	\end{proof}
	
	\section{A Non-Simplicial Example in Dimension 3}
	
	To illustrate the framework beyond the simplicial case, we examine a projective non-simplicial toric $3$-fold. The Betti numbers and the spanning property can be verified directly via the toric $h$-vector and Fieseler's surjectivity theorem.
	
	\subsection{The Hexagonal Pyramid Fan}
	
	Let $P$ be the hexagonal pyramid in $\IR^3$: the convex hull of the six lattice points
	\[
	v_1=(1,0,0),\; v_2=(0,1,0),\; v_3=(-1,1,0),\; v_4=(-1,0,0),\; v_5=(0,-1,0),\; v_6=(1,-1,0)
	\]
	and the apex $v_7=(0,0,1)$. The normal fan $\Sigma = \Sigma_P \subset \IR^3$ has one ray for each facet of $P$:
	\begin{align*}
		\rho_b &= (0,0,-1) \quad \text{(hexagonal base)}, \\
		\rho_1 &= (1,1,1),\quad \rho_2 = (0,1,1),\quad \rho_3 = (-1,0,1), \\
		\rho_4 &= (-1,-1,1),\quad \rho_5 = (0,-1,1),\quad \rho_6 = (1,0,1)
	\end{align*}
	(the six triangular side faces). There are seven maximal cones, one for each vertex of $P$:
	\[
	\sigma_j = \mathrm{cone}(\rho_b, \rho_j, \rho_{j-1}) \quad (j=1,\dots,6),
	\]
	corresponding to the six base vertices (where $\rho_0$ is interpreted cyclically as the side ray $\rho_6$), and
	\[
	\sigma_7 = \mathrm{cone}(\rho_1, \rho_2, \rho_3, \rho_4, \rho_5, \rho_6),
	\]
	corresponding to the apex $v_7$. The cones $\sigma_1,\dots,\sigma_6$ are simplicial (each has $\det = 1$, so $X_\Sigma$ is smooth along the corresponding orbits). The cone $\sigma_7$ is \emph{non-simplicial}: it has six rays in $\IR^3$ and corresponds to the unique isolated singular point of $X_\Sigma$.
	
	\subsection{IH Betti Numbers via the Toric $h$-Vector}
	
	For a projective toric variety associated to a rational polytope $P$, the intersection-cohomology Betti numbers are exactly given by the generalized (toric) $h$-vector introduced by Stanley \cite{Stanley87}:
	\[
	\dim IH^{2k}(X_P, \IQ) = h_k(P).
	\]
	For a $3$-dimensional polytope $P$, the toric $h$-vector satisfies:
	\[
	h_0 = 1,\qquad h_1 = f_2(P) - 3,\qquad h_2 = h_1,\qquad h_3 = 1,
	\]
	where $f_2(P)$ is the number of facets of $P$ \cite[Theorem 2.3]{Stanley87}. Since the hexagonal pyramid has $f_2(P) = 7$ facets (one hexagonal base and six triangular sides), we compute directly:
	\[
	h_{\mathrm{toric}}(P) = (1, 4, 4, 1).
	\]
	Therefore, the IH Betti numbers are uniquely determined independently of any cycle map:
	\[
	\dim IH^0 = 1,\quad \dim IH^2 = 4,\quad \dim IH^4 = 4,\quad \dim IH^6 = 1.
	\]
	
	By Fieseler \cite[Theorem 1.1]{Fieseler91}, the cycle class map $\Chowh{2}(X_\Sigma)_\IQ \to IH^2(X_\Sigma, \IQ)$ from the Chow homology group of divisor cycles is surjective. For a complete toric variety, the rank of the Weil divisor class group $\Chowh{2}(X_\Sigma)_\IQ$ equals the number of rays minus the dimension of the lattice:
	\[
	\operatorname{rank} \Chowh{2}(X_\Sigma)_\IQ = 7 - 3 = 4.
	\]
	Since $\operatorname{rank} \Chowh{2}(X_\Sigma)_\IQ = 4$ and $\dim IH^2(X_\Sigma, \IQ) = 4$, Fieseler's surjection is in fact an isomorphism for this example.
	
	\subsection{Divisor Relations and Basis for $\Chowh{2}(X_\Sigma)_\IQ$}
	
	The seven torus-invariant Weil divisors $D_b, D_1, \dots, D_6$ (one per ray) satisfy three linear equivalence relations in the Chow homology group $\Chowh{2}(X_\Sigma)_\IQ$, one for each lattice character $e_1, e_2, e_3 \in M = \IZ^3$:
	\begin{align*}
		\langle e_1, \cdot \rangle: &\quad D_1 - D_3 - D_4 + D_6 = 0, \\
		\langle e_2, \cdot \rangle: &\quad D_1 + D_2 - D_4 - D_5 = 0, \\
		\langle e_3, \cdot \rangle: &\quad -D_b + D_1 + D_2 + D_3 + D_4 + D_5 + D_6 = 0.
	\end{align*}
	Solving for $D_5$, $D_6$, $D_b$ in terms of $\{D_1, D_2, D_3, D_4\}$:
	\[
	D_5 = D_1 + D_2 - D_4, \quad D_6 = D_3 + D_4 - D_1, \quad D_b = D_1 + 2D_2 + 2D_3 + D_4.
	\]
	Hence $\Chowh{2}(X_\Sigma)_\IQ \cong \IQ^4$ with basis $\{[D_1],[D_2],[D_3],[D_4]\}$, confirming the rank computation above.
	
	\subsection{Illustration of the Spanning Property}
	
	\textbf{Degree 2.} As established by Fieseler's isomorphism, the four divisor-cycle basis classes $[D_1]_{\mathrm{comb}}, [D_2]_{\mathrm{comb}}, [D_3]_{\mathrm{comb}}, [D_4]_{\mathrm{comb}}$ precisely span the $4$-dimensional space $\IHcomb^2(\Sigma, \IQ)$.
	
	\textbf{Degree 4.} Theorem~\ref{thm:dim3} establishes the generation statement for all projective toric $3$-folds. This example serves as an illustration of the general mechanism: the $4$-dimensional space $\IHcomb^4(\Sigma, \IQ)$ is spanned by the twelve combinatorial curve classes $[V(\tau_{ij})]_{\mathrm{comb}}$ corresponding to the twelve edges of $P$. The ample operational class $c_1(L) \in \Chowop{1}(X_\Sigma)$ acts on the four divisor-cycle classes in $\Chowh{2}(X_\Sigma)$, producing linear combinations of curve classes in $\Chowh{1}(X_\Sigma)$. The Hard Lefschetz isomorphism $L: \IHcomb^2(\Sigma) \xrightarrow{\sim} \IHcomb^4(\Sigma)$ maps the four divisor basis classes to four linearly independent elements of $\IHcomb^4(\Sigma, \IQ)$, confirming that these curve classes span the intersection cohomology in degree $4$.
	
	\section{Conclusion}
	
	We have formulated a precise linear-generation statement for the combinatorial cycle classes within the BBFK intersection cohomology theory. Under the BBFK--BL functoriality assumptions stated in Assumption~\ref{prop:compatibility}, the proposed Gysin construction is compatible with geometric cycle classes. Leveraging operational Chow cohomology and Chow homology, we proved the conjecture for dimensions up to~$3$ and for simplicial fans. The non-simplicial example of the hexagonal pyramid toric variety illustrates the framework, with the IH Betti numbers $1,4,4,1$ derived directly from Stanley's toric $h$-vector formula. The general case for dimension $n \ge 4$ remains an open problem, closely tied to the surjectivity of the map from Minkowski weights to intersection cohomology.
	
	\section*{Statements and Declarations}
	
	\noindent\textbf{Competing interests.}
	The authors declare that they have no competing interests.
	
	\noindent\textbf{Funding.}
	No funding was received for this work.
	
	\noindent\textbf{Data availability.}
	Data sharing is not applicable to this article as no datasets were generated or analyzed.
	

\end{document}